\documentclass[11pt]{article}

 \usepackage{hyperref}

\usepackage{amsmath,amsthm,amssymb,xypic,graphicx}

\theoremstyle{plain}
    \newtheorem{theorem}{Theorem}[section]
    \newtheorem{lemma}[theorem]{Lemma}
    
    \newtheorem{proposition}[theorem]{Proposition}

 \theoremstyle{definition}
    \newtheorem{definition}[theorem]{Definition}
    
    \newtheorem{example}[theorem]{Example}

    \newtheorem{remark}[theorem]{Remark}

\theoremstyle{remark}

\numberwithin{equation}{section}

\DeclareMathOperator{\Ad}{Ad}

\DeclareMathOperator{\ind}{index}

\DeclareMathOperator{\reg}{reg}

\DeclareMathOperator{\Spin}{Spin}

\DeclareMathOperator{\SO}{SO}
\DeclareMathOperator{\SL}{SL}
\DeclareMathOperator{\GL}{GL}
\DeclareMathOperator{\SU}{SU}
\DeclareMathOperator{\U}{U}
\DeclareMathOperator{\PGL}{PGL}

\DeclareMathOperator{\Inn}{Inn}
\DeclareMathOperator{\Gal}{Gal}
\DeclareMathOperator{\Aut}{Aut}

\begin{document}


\newcommand{\myemph}{\emph}

\newcommand{\Spinc}{\Spin^c}
\newcommand{\LanG}{^L\!G}
    \newcommand{\R}{\mathbb{R}}
    \newcommand{\C}{\mathbb{C}} 
    \newcommand{\N}{\mathbb{N}}
    \newcommand{\Z}{\mathbb{Z}} 
    \newcommand{\Q}{\mathbb{Q}}
    \newcommand{\bT}{\mathbb{T}}
    \newcommand{\bP}{\mathbb{P}}

\newcommand{\g}{\mathfrak{g}}
\newcommand{\h}{\mathfrak{h}}
\newcommand{\p}{\mathfrak{p}}
\newcommand{\kg}{\mathfrak{g}} 
\newcommand{\kt}{\mathfrak{t}}
\newcommand{\ka}{\mathfrak{a}}
\newcommand{\XX}{\mathfrak{X}}
\newcommand{\kh}{\mathfrak{h}} 
\newcommand{\kp}{\mathfrak{p}}
\newcommand{\kk}{\mathfrak{k}}

\newcommand{\cA}{\mathcal{A}}
\newcommand{\cE}{\mathcal{E}}
\newcommand{\calL}{\mathcal{L}}
\newcommand{\calH}{\mathcal{H}}
\newcommand{\cO}{\mathcal{O}}
\newcommand{\cB}{\mathcal{B}}
\newcommand{\cK}{\mathcal{K}}
\newcommand{\cP}{\mathcal{P}}
\newcommand{\cN}{\mathcal{N}}
\newcommand{\calD}{\mathcal{D}}
\newcommand{\cC}{\mathcal{C}}
\newcommand{\calS}{\mathcal{S}}
\newcommand{\cM}{\mathcal{M}}

\newcommand{\cCM}{\cC}
\newcommand{\PM}{P}
\newcommand{\DM}{D}
\newcommand{\LM}{L}
\newcommand{\vM}{v}

\newcommand{\LG}{L^G_{\lambda - \rho}}
\newcommand{\LGc}{L^{G_c}_{\lambda - \rho}}

\newcommand{\Bigwedge}{\textstyle{\bigwedge}}

\newcommand{\ii}{\sqrt{-1}}

\newcommand{\beq}[1]{\begin{equation} \label{#1}}
\newcommand{\eeq}{\end{equation}}

\newcommand{\ddt}{\left. \frac{d}{dt} \right|_{t=0}}

\newcommand{\Todo}{\textbf{To do}}

\newenvironment{proofof}[1]
{\noindent \emph{Proof of #1.}}{\hfill $\square$}

\title{Shelstad's character identity from the  point of view of index theory}

\author{Peter Hochs\footnote{University of Adelaide, \texttt{peter.hochs@adelaide.edu.au}} \ 
and Hang Wang\footnote{University of Adelaide, \texttt{hang.wang01@adelaide.edu.au}}}


\date{\today}

\maketitle

\begin{abstract}
Shelstad's character identity is an equality between sums of characters of tempered representations in corresponding $L$-packets of two real, semisimple, linear, algebraic groups that are inner forms to each other.
We reconstruct  this character identity in the case of the discrete series, using index theory of elliptic operators in the framework of $K$-theory.
Our geometric proof of the character identity 
is evidence that index theory can play a role in the classification of group representations via the Langlands program.
\end{abstract}

\tableofcontents

\section{Introduction}

The aim of this paper is to use our previous results in~\cite{Hochs-Wang} to compare representations of two connected real semisimple groups $G, G'$ 
having the same Langlands dual group.
We are inspired by the so-called Shelstad's character identity following from Langlands program~\cite{ShelstadThesis, Shelstad79}. 
In fact, in the local Langlands program~\cite{Langlands1, Langlands2, Langlands3}, every admissible representation of a real reductive Lie group is labeled by its $L$-parameter, which is represented by a  homomorphism from the Weil group $W_{\R}$ of the real numbers to the Langlands dual group of $G.$ 
If two real reductive groups $G$ and $G'$ have the same Langlands dual group, then every $L$-parameter $\phi$ of $G$ can be identified to an $L$-parameter $\phi'$ for $G'$. 
Denote by $\Pi_{\phi}$ the $L$-packet of $\phi$, i.e., the set of admissible representations of $G$ having the same $L$-parameter $\phi$.
Shelstad's character identity states that the characters $\Theta_{\pi}$ of tempered representations $\pi$ for $G$ and $G'$ associated to the same $L$-parameter satisfy the identity 
\begin{equation}
\label{eq:Shelstad}
(-1)^{\dim G/K}\sum_{\pi\in\Pi_{\phi}}\Theta_{\pi}(h)=(-1)^{\dim G'/K'}\sum_{\pi'\in\Pi_{\phi'}}\Theta_{\pi'}(h'), 
\end{equation}
if $h$ and $h'$ are corresponding regular elements of Cartan subgroups $H<G$ and $H'<G'$, respectively.



The main result of this paper is a direct geometric proof of the character identity~(\ref{eq:Shelstad}) in the case of discrete series representations, without referring to character formulas or the theory of Langlands program. The proof uses index theory and $K$-theory of $C^*$-algebras, and illustrates how these are related to representation theory and the Langlands program.

\subsection*{Acknowledgements}

The authors are grateful to Kuok Fai Chao for background information about the Langlands program. The second author was supported by the Australian Research Council, through Discovery Early Career
Researcher Award DE160100525.

\section{Character identities and index theory}

\subsection{Inner forms}

We recall the definition of inner forms, as discussed for example in~\cite{ABV}. 
Let $G_{\C}$ be a connected reductive complex algebraic group. A \emph{real form} of $G_{\C}$ is an involutive automorphism $\sigma\colon G_{\C}\rightarrow G_{\C}$ which is antiholomorphic; i.e., $T_e\sigma(iX)=-iT_e\sigma(X)$ for all $X \in \kg_{\C}$. 
The involution $\sigma$ is a generalised notion of complex conjugation.
The \emph{group of real points} of $\sigma$ 
\[
G(\R, \sigma)=\{g\in G_{\C}: \sigma(g)=g\}
\]
is called the \emph{real form} of $G_{\C}$ associated to $\sigma.$

\begin{definition}[{\cite[Chapter~2]{ABV}}]
Two real forms $G(\R, \sigma)$ and $G(\R, \sigma')$ of $G_{\C}$ are said to be \emph{inner} to each other if there is an element $g\in G_{\C}$ such that 
\begin{equation}
\label{eq:inner.form.relation}
\sigma'=C_g\circ\sigma,
\end{equation} 
where $C_g$ denotes conjugation by $g$.
The real group $G':=G(\R, \sigma')$ is called an \emph{inner form} of the real group $G:=G(\R, \sigma)$.
\end{definition}


The set of inner forms for $G$ can be identified bijectively to the first Galois cohomology $H^1(\Gamma, \Inn(G_{\C}))$ where $\Gamma$ is the Galois group $\Gal(\C/\R)$ of two elements, where the complex conjugation (the generator) takes $\eta\in \Aut(G_{\C})$ to $\sigma\circ\eta\circ\sigma^{-1}\in \Aut(G_{\C})$, and $\Inn(G_{\C})$ is the subgroup of inner automorphisms of $G_{\C}$ respecting the action of $\Gamma$. See~\cite{Borel}.

\begin{example}
\label{ex:inner.form}
\begin{itemize}
\item Let $G_{\C}=\SL(2,\C)$. $G_{\C}$ has a compact real form $G=\SU(2)$ associated to $\sigma(g)=(\bar g^T)^{-1}$ and a split real form $G'=\SL(2, \R)$ accosiated to $\sigma(g)=\bar g.$ It can be checked that $G'$ is the unique inner form for $G$.
\item $\SL(n, \R)$ when $n$ odd does not have any inner form other than itself; It has an extra inner form if $n$ is even.
Note that when $n\ge 3$, $\SL(n, \R)$ does not have compact inner forms in view of Theorem~\ref{Harish-Chandra} below.
\item $\U(2)$ and $\U(1,1)$ are inner forms to each other. 

\item $\SU(2,1)$ and $\SU(3)$ are inner forms to each other.
\item $\SO(2p, 2q)$ where $2p+2q=4n$ are inner forms to each other. In particular, $\SO(4n)$ is the compact inner form amongst elements of this set.
\end{itemize}
\end{example}

If $G$ and $G'$ are inner forms to each other in a complex group $G_{\C}$, then any Cartan subgroup $H<G$ is conjugate in $G_{\C}$ to a Cartan subgroup $H'<G'$. See~Lemma~2.1 of~\cite{Langlands3} and~\cite{Shelstad79}. We will use this to identify elements $h \in H$ to corresponding elements $h' \in H'$.
\begin{example}
Let $G=\SU(2)$ and $G'=\SL(2, \R)$. Then their respective Cartan subgroups 
\[
T=\left\{\begin{bmatrix}e^{i\theta} & 0 \\ 0 & e^{-i\theta}\end{bmatrix}:\theta\in\R\right\} \qquad\text{and}\qquad T'=\left\{\begin{bmatrix}\cos\theta & \sin\theta \\ -\sin\theta & \cos\theta \end{bmatrix}:\theta\in\R\right\}
\]
are conjugated by $\begin{bmatrix}-i & -1 \\ \frac12 & \frac{i}{2}\end{bmatrix}\in \SL(2, \C).$
\end{example}

\subsection{$L$-packets and character identities}

Inner forms are closely related to the Langlands program. Let $G$ be a real reductive algebraic group, and let $G^\vee$ denote the Langlands dual group of $G$. Consider the \emph{$L$-group} $\LanG = G^\vee \rtimes \Gal(\C/\R)$. Let  $W_{\R}$ be the \emph{Weil group} of real numbers, i.e., the group $\C^{\times}\cup j\C^{\times}$ subject to relations
\[
j^2=-1 \quad \text{and} \quad jz=\overline{z}j. 
\]
Then the irreducible admissible representations of $G$  are parametrised by $L$-parameters. These are group homomorphisms 
\[
\phi\colon W_{\R}\rightarrow \LanG
\] 
satisfying certain conditions.
In fact, the local Langlands correspondence gives rise to a surjective map 
\begin{equation}\label{eq:LLC}
f: \Pi(G)\rightarrow \Phi(G) \qquad \text{where}
\end{equation}
\begin{align*} 
\Pi(G)&:=\{\text{equivalence classes of irreducible admissible representations of }G\}; \\
\Phi(G)&:=\{\text{$G^\vee$-conjugacy classes of $L$-parameters } \phi\colon W_{\R}\rightarrow\LanG\}
\end{align*}
where $f^{-1}(\{\phi\})$ is finite for every $\phi\in \Phi(G).$
The finite set 
\[
\Pi_{\phi}:=f^{-1}(\{\phi\})
\]
is called the \emph{$L$-packet} associated to the $L$-parameter $\phi$.
In other words, two admissible irreducible representations $\pi_1, \pi_2$ are said to be in the same $L$-packet if and only if they have the same $L$-parameter. 
Representations having the same $L$-parameter are indistinguishable in the sense of Langlands.

When $G$ and $G'$ are inner forms to each other, they share the same $L$-group. 
Groups with the same  $L$-group are crucial in the sense that the representation theory of these groups are closely related and can be studied globally.
The character identity in Theorem \ref{thm Shelstad} below, regarding inner forms is an instance of this philosophy. 
In fact, by the local Langlands correspondence, the $L$-packets of admissible representations for $G$ and $G'$ are related by their corresponding $L$-parameters. 
Hence, 
a character identity involving a common $L$-packet can be expected. 
For more details of Langlands program and character identity we refer to Langlands~\cite{Langlands1, Langlands2, Langlands3} and Shelstad~\cite{ShelstadThesis, Shelstad79}.



Let $\phi$ be an $L$-parameter of admissible representations of $G$. 
According to Langlands~\cite{Langlands3}, all admissible representations in the same $L$-packet of a tempered representation are tempered. Hence the $L$-parameter $\phi$ of a tempered representation is called tempered. 
Let $G'$ be 
 an  inner form of $G$.
 Suppose $G$ is 
 \emph{quasi-split}, meaning that $G$  has a $\sigma$-invariant Borel subgroup.
  In that case, we have $\Phi(G')\subset \Phi(G)$. Suppose $\phi \in \Phi(G')$. We write $\phi'$ for $\phi$, when we view it as an $L$-parameter of $G'$.
Then $\phi'$ is tempered if $\phi$ is.
The same line of statements remains true with ``tempered' replaced
 by ``discrete series".  

For any irreducible admissible representation $\pi$, we denote its global character by $\Theta_\pi$.
Shelstad's character identity is stated as follows.
\begin{theorem}[Shelstad's character identity~\cite{ShelstadThesis, Shelstad79}] \label{thm Shelstad}
Let $G$ and $G'$ be connected, real reductive, linear algebraic groups, and inner forms to each other in $G_{\C}$. Suppose $G$ is quasi-split. Let $K<G$ and $K'<G'$ be maximal compact subgroups and let $H<G$ and $H'<G$ be Cartan subgroups that are conjugate to each other in $G_{\C}$.
Let $h$ be an elliptic regular element of $H$ and $h'$ the corresponding element of $H'$. 
Let $\phi'$ be a tempered $L$-parameter of $G'$, and let 
$\phi$ be the corresponding $L$-parameter of $G$.
Then 
\begin{equation}
\label{eq:Char.id}
(-1)^{\frac{\dim G/K}{2}}\sum_{\pi\in\Pi_{\phi}}\Theta_{\pi}(h)=(-1)^{\frac{\dim G'/K'}{2}}\sum_{\pi'\in\Pi_{\phi'}}\Theta_{\pi'}(h').
\end{equation}
\end{theorem}
This is Theorem~6.3 in~\cite{Shelstad79}.
In this paper, we will prove Shelstad's character identity~(\ref{eq:Char.id}) using $K$-theory and index theory when $\Pi_{\phi}$ and $\Pi_{\phi'}$ are $L$-packets of discrete series representations. 

\subsection{Character identities for the discrete series}

We now assume in addition that the group $G$ is semisimple, and has discrete series representations. Harish-Chandra has the following equivalent statements for existence of discrete series. See for example Theorem 22.1 in~\cite{Thorne}.

\begin{theorem}[Harish-Chandra]
\label{Harish-Chandra}
The following statements are equivalent.
\begin{enumerate}
\item $G$ has a compact Cartan subgroup $T<G$;
\item $G$ has a compact inner form $G'$;
\item $G$ has discrete series, i.e., square integrable representations.
\end{enumerate}
In particular, if $G$ has a discrete series, then it always has a compact inner form.
\end{theorem}

As a result, all but the second case in Example~\ref{ex:inner.form} admit discrete series and hence they have compact Cartan subgroups.
Assume $G$ has discrete series and $T$ is a compact Cartan subgroup from now on. 
\begin{theorem}[{\cite[Corollary 2.9]{Shelstad79}}]
Any inner form $G'$ of $G$ also 
 contains a compact Cartan subgroup.
\end{theorem}
So an inner form $G'$ of $G$ also has discrete series representations. 



For discrete series, we have the following property, of $L$-packets, which can be found for example in~\cite{Labesse}.
\begin{proposition}\label{prop Lpack ds}
If $G$ has discrete series, then two discrete series representations $\pi_1$ and $\pi_2$ are in the same $L$-packet if and only if they have the same infinitesimal character. 
\end{proposition}
Let $K<G$ be a maximal compact subgroup containing $T$.
Let $W_G$ and $W_K <W_G$ be the Weyl groups of the root systems of $(\kg^{\C}, \kt^{\C})$ and $(\kk^{\C}, \kt^{\C})$, respectively. Then
 two discrete series representations of $G$ have the same infinitesimal character if and only if their Harish-Chandra parameters are in the same $W_G$-orbit. They are equivalent if and only if they are in the same $W_K$-orbit, so every discrete series $L$-packet can be identified with $W_G/W_K$.
 
 In this paper, we will give a geometric proof of a special case of Theorem \ref{thm Shelstad}
 \begin{theorem} \label{thm char id}
 Theorem \ref{thm Shelstad} holds for $L$-packets of discrete series representations.
 \end{theorem}
 This case of Theorem \ref{thm Shelstad} can be proved by explicitly writing out Harish--Chandra's character formula for the discrete series and rearranging terms. Our proof does not involve character formulas, and will show how this case of the character identities is related to the geometry of the space $G/T$, and also how character identities are related to $K$-theory and index theory.

\begin{example}\label{ex L param SL2}
Let $G=\SL(2,\R)$, and  $K=T = \SO(2)$.   Let $\rho \in i\kt^*$ be the element mapping $\begin{bmatrix} {0} & {-1} \\{1} &{0} \end{bmatrix} $ to $i$. Let $n \in \N$, and set $\lambda = n\rho$. 
Let $\pi^{G}_{\pm\lambda}$ be the discrete series representation of $\SL(2, \R)$ with Harish-Chandra parameter $\pm\lambda$.

We have $G^{\vee} = \PGL(2,\C)$. The $L$-parameter of the two representations $\pi^{G}_{\pm\lambda}$, which are in the same $L$-packet, is the homomorphism
\[
\phi_n \colon W_{\R}\to \PGL(2,\C)\rtimes \Gal(\C/\R)
\]
defined as follows. Let $\sigma \in \Gal(\C/\R)$ be the nontrivial element, i.e.\ complex conjugation. 
For $r>0$ and $\theta \in \R$, set $\chi_n(re^{i\theta}) = re^{in\theta}$. Then 
\[
\begin{split}
\phi_n(j) &= \Bigl( \begin{bmatrix} 0 & (-1)^n \\ 1 & 0\end{bmatrix}, \sigma \Bigr)\\
\phi_n(z) &=  \Bigl( \begin{bmatrix} \chi_n(z) & 0 \\ 0 & \chi_{-n}(z)\end{bmatrix}, e\Bigr),
\end{split}
\]
for $z\in \C^{\times}$. (See Section I.4 of \cite{Labesse}.)
\end{example}

\subsection{An equivariant index and orbital integrals}
Let $G$ be a connected, real semisimple Lie group with finite centre. 
Suppose $G$ acts properly and isometrically on a Riemannian manifold $M$. Suppose $M/G$ is compact. Let $E \to M$ be a $\Z_2$-graded, $G$-equivariant, Hermitian vector bundle. Let $D$ be a $G$-equivariant, elliptic, self-adjoint first order differential operator on $E$ that is odd with respect to the grading. The \emph{reduced group $C^*$-algebra} $C^*_rG$ of $G$ is the closure in the operator norm of the algebra of all convolution operators on $L^2(G)$ by functions in $L^1(G)$.
Let  $K_0(C^*_rG)$ be its even $K$-theory. Then we have the equivariant index of $D$
\[
\ind_G(D) \in K_0(C^*_rG),
\]
which is the image of the class defined by $D$ in the equivariant $K$-homology group of $M$ under the analytic assembly map. See \cite{Connes94} for details.

Let $g \in G$ be a semisimple element, and $Z<G$ its centraliser. The orbital integral map
\[
f \mapsto \int_{G/Z}f(hgh^{-1})\, d(hZ)
\]
on $C_c(G)$, defines a trace map
\[
\tau_g \colon K_0(C^*_rG) \to \C.
\]
See Section 2.1 in \cite{Hochs-Wang}.

Theorem 2.1 in \cite{Hochs-Wang} is a fixed point formula for the number
\[
\tau_g(\ind_G(D)).
\]
We will not use this fixed point formula, but only the following localisation or excision property of the index.
This property follows directly from Theorem 2.1 in \cite{Hochs-Wang}, but is in fact a step in its proof: see Proposition 4.6 in \cite{Hochs-Wang}.
\begin{proposition}[Localisation of the index] \label{prop loc}
Let $G'$, $M'$, $E'$ and $D'$ be as $G$, $M$, $E$ and $D$ above, respectively. Suppose that $G$ and $G'$ are subgroups of some larger group, and that $g \in G \cap G'$. Suppose that there are $g$-invariant neighbourhoods $U$ of the fixed point set $M^g$ and $U'$ of $(M')^g$, and a diffeomorphism $\varphi\colon U \to U'$ that commutes with $g$, such that
\[
\varphi^*(E'|_{U'}) = E|_U,
\]
and $D'|_{U'}$ corresponds to $D|_U$ under this identification. Then
\[
\tau_g(\ind_G(D)) = \tau_g(\ind_{G'}(D')).
\]
\end{proposition}

\subsection{Discrete series characters as indices}

Let $K<G$ be maximal compact. Let $T<K$ be a maximal torus and suppose that $T$ is a Cartan subgroup of $G$; i.e.\ $G$ has a discrete series. Let 
\[
R := R(\kg^{\C}, \kt^{\C}).
\]
be the root system of $(\kg^{\C}, \kt^{\C})$.
Choose a positive system $R^+\subset R$, and let $\rho$ be half the sum of the elements of $R^+$. Let $\lambda \in i\kt^*$ be regular, and dominant with respect to $R^+$. Suppose $\lambda - \rho$ is integral. Let $\pi^G_{\lambda}$ be the discrete series representation of $G$ with Harish-Chandra parameter $\lambda$. The values of its character $\Theta_{\pi^G_{\lambda}}$ on the regular elements of $T$ can be realised in terms of index theory.

%
%

Consider the $G$-manifold $G/T$, equipped with the $G$-invariant Riemannian metric defined by a $K$-invariant inner product on $\kg$. Consider the $G$-invariant complex structure $J^{G/T}_{R^+}$ on $G/T$ such that, as complex vector spaces,
\beq{eq def J}
T_{eT}(G/T) = \kg/\kt \cong \bigoplus_{\alpha \in R^+} \kg^{\C}_{\alpha}.
\eeq
Explicitly, this isomorphism is given by the inclusion $\kg/\kt \hookrightarrow (\kg/\kt)^{\C}$ followed by projection onto the positive root spaces. Similarly, if $w \in W_G$, then we will write $J^{G/T}_{wR^+}$ for the complex structure defined as above, with $R^+$ replaced by $wR^+$. 

For any integral element $\nu \in i\kt^*$, we have the $G$-equivariant line bundle
\[
L^G_{\nu} := G\times_{T}\C_{\nu}\to G/T.
\]
Here we write $\C_{\nu}$ for the vector space $\C$ on which $T$ acts with weight $\nu$. Let $\bar \partial_{\LG}$ be the Dolbeault operator on $G/T$ coupled to $\LG$. Let $\bar \partial_{\LG}^*$ be its formal adjoint with respect to the $L^2$-inner product defined by the Riemannian metric on $G/T$ and the natural Hermitian metric on $\LG$.
\begin{proposition}\label{prop ds index}
If  $g \in T^{\reg}$, then
\[
\Theta_{\pi^G_{\lambda}}(g) = (-1)^{\dim(G/K)/2}\tau_g(\ind_G(\bar \partial_{\LG}+\bar \partial_{\LG}^*)).
\]
\end{proposition}
For a proof,  see Propositions 5.1 and 5.2 in \cite{Hochs-Wang}. This proof is based on the fact that the natural class $[\pi_{\lambda}^G] \in K_0(C^*_rG)$ (see e.g.\  \cite{Lafforgue02}) equals 
\[
(-1)^{\dim(G/K)/2}\ind_G(\bar \partial_{\LG}+\bar \partial_{\LG}^*).
\]

\section{A geometric proof} \label{sec proof}

Let $G$ be a connected, real semisimple Lie group with finite centre. Let $K<G$ be a maximal compact subgroup, and suppose a maximal torus $T<K$ is a Cartan subgroup of $G$. As before, Let 
$
R 
$
be the root system of $(\kg^{\C}, \kt^{\C})$,  let 
 $R^+\subset R$ be a positive system, and let $\rho$ be half the sum of the elements of $R^+$. 
Since $G$ has a compact Cartan subgroup, it has a compact inner form $G_c$ by Theorem \ref{Harish-Chandra}. Theorem \ref{thm char id} follows from the case where $G'=G_c$, which we will assume from now on. The complexifications $\kg^{\C}$ and $\kg_c^{\C}$ of the Lie algebras $\kg$ and $\kg_c$ of $G$ and $G_c$, respectively, are equal. So if we identify $T$ with a $G_{\C}$-conjugate Cartan subgroup $T_c <G_c$, then 
the root system of $(\kg_c^{\C}, \kt_c^{\C})$ equals $R$. From now on, we  will tacitly identify $T$ and $T_c$ with each other.

\subsection{Identifying open sets} \label{sec nghbhd}

The geometric proof of Theorem \ref{thm char id} is based on suitable identifications of neighbourhoods of the fixed point sets of the actions by $T$ on $G/T$ and $G_c/T$.

Let $w \in N_{G_c}(T)$. Consider the $T$-invariant complex structure 
on $\kg/\kt$ defined by ${w^{-1}R^+}$ and the one on $\kg_c/\kt$ defined by ${R^+}$. 
Let $\kt^{\perp} \subset \kg$  be the orthogonal complement to $\kt$ in $\kg$ with respect to an $\Ad(K)$-invariant inner product. Similarly, let $\kt^{\perp_c} \subset \kg_c$ be the orthogonal complement to $\kt$ in $\kg_c$. 
Then we  have a complex-linear isomorphism
\beq{eq gt gct}
\psi_w\colon
\kt^{\perp} \cong
\kg/\kt \cong \bigoplus_{\alpha \in R^+} \kg^{\C}_{w^{-1}\alpha} 
\xrightarrow{\Ad(w^{-1})} \\
\bigoplus_{\alpha \in R^+} \kg^{\C}_{\alpha} \cong \kg_c/\kt
 \cong \kt^{\perp_c}.
\eeq
Note that $\Ad(w)$ preserves $\kt^{\perp_c}$.
\begin{lemma}\label{lem Adw psi equivar}
The composition $\Ad(w) \circ \psi_w\colon \kt^{\perp} \to \kt^{\perp_c}$ is $T$-equivariant. 
\end{lemma}
\begin{proof}
Let $X \in \kt^{\perp}$. For any choice $\tilde R^+$ of positive roots, let  
\[
p_{\tilde R^+}\colon \kg^{\C}/\kt^{\C} \to  \bigoplus_{\alpha \in \tilde R^+} \kg^{\C}_{\alpha} 
\]
be the projection map defined by the root space decomposition. Then
\[
\Ad(w)\circ p_{R^+} = p_{w^{-1}R^+}\circ \Ad(w).
\]

By definition of $\psi_w$,
\[
p_{R^+}(\psi_w(X)) = \Ad(w^{-1})p_{w^{-1}R^+}(X).
\]
Hence
\[
 p_{w^{-1}R^+}( \Ad(w)\psi_w(X)) = \Ad(w)p_{R^+}(\psi_w(X)) = p_{w^{-1}R^+}(X).
\]
Since $p_{w^{-1}R^+}$ is $T$-equivariant, we have for all $h \in T$,
\[
 p_{w^{-1}R^+}(\Ad(h) \Ad(w)\psi_w(X)) = p_{w^{-1}R^+}(\Ad(h)X) = p_{w^{-1}R^+}( \Ad(w)\psi_w(\Ad(h)X)).
\]
So $\Ad(w)\circ \psi_w$ is indeed $T$-equivariant.
\end{proof}

The map
$\eta\colon \kt^{\perp}\to G/T$
mapping $X \in \kt^{\perp}$ to $\exp(X)T$ is $T$-equivariant, and a local diffeomorphism near $0 \in \kt^{\perp}$. 
Let $\tilde U \subset \kt^{\perp}$ be an $\Ad(T)$-invariant open neighbourhood of $0$ on which this map defines a diffeomorphism onto its image $U \subset G/T$. The tangent map of $\eta$ at $0$ is the identification $\kt^{\perp} \cong T_{eT}G/T$, which is complex-linear by definition. So $\eta^*(J^{G/T}_{w^{-1}R^+}|_{U})$ is homotopic to $J_{w^{-1}R^+}|_{\tilde U}$ as $T$-invariant complex structures, if we choose $\tilde U$ small enough. So $\eta$ is holomorphic up to a homotopy of $T$-invariant almost complex structures, which is as good as being holomorphic for index theory purposes.

Let $\eta_c\colon \kt^{\perp_c} \to G_c/T$ be defined by $\eta_c(Y) = \exp(Y)T$ for $Y \in \kt^{\perp_c}$. This is a $T$-equivariant map, and a local diffeomorphism near $0$. Choose $\tilde U$ small enough, so that
 $\eta_c|_{\tilde \psi_w(\tilde U)}$ is a diffeomorphism onto its image $U_c^w$, which is a neighbourhood of $eT$ in $G_c/T$. Again, this diffeomorphism is 
 holomorphic up to a $T$-equivariant homotopy of complex structures if we choose $\tilde U$ small enough.
For later use, we define the  map $\eta_c^w\colon \kt^{\perp_c}\to G_c/T$  by $\eta_c^w(Y) = \exp(Y)wT$, for $Y \in \kt^{\perp_c}$. This map is $T$-equivariant because $w$ normalises $T$.  
 
 The maps $\psi_w$, $\eta$ and $\eta_c$ combine into a  diffeomorphism
\[
\varphi_w\colon U \xrightarrow{\eta^{-1}} \tilde U \xrightarrow{\psi_w} \psi_w(U) \xrightarrow{\eta_c} U_c^w,
\]
holomorphic up to a homotopy of complex structures. The map $\psi_w$ is not $T$-equivariant, so neither is $\varphi_w$ in general.
\begin{lemma}\label{lem T equivar}
The composition
\[
w\circ \varphi_w\colon U \to wU_c^w
\]
is $T$-equivariant.
\end{lemma}
\begin{proof}
We have
\[
w\circ \eta_c = \eta_c^w \circ \Ad(w).
\]
Hence
\[
w\circ \varphi_w = \eta_c^w \circ (\Ad(w) \circ \psi_w) \circ \eta^{-1}.
\]
The maps $\eta_c^w$ and $\eta$ are $T$-equivariant, and by
Lemma \ref{lem Adw psi equivar}, so is
$\Ad(w) \circ \psi_w$. So  $w\circ \varphi_w$ is $T$-equivariant as well.
\end{proof}

%



\subsection{Neighbourhoods of fixed point sets}

Let $W_K := N_K(T)/T$ and $W_G := N_{G_c}(T)/T$. Then $W_G$ is the Weyl group of the root system $R$, which explains the notation $W_G$. And $W_K$ is the subgroup of $W_G$ generated by reflections defined by compact roots. 
Note that
\[
(G/T)^T = W_K 
\]
(see Lemma 6.7 in \cite{Hochs-Wang.AK}), and
\[
(G_c/T)^T =  W_G.
\]
Hence we have an inclusion
\[
(G/T)^T \hookrightarrow (G_c/T)^T 
\]
and
\[
(G_c/T)^T = \coprod_{[w] \in W_G/W_K} w\cdot (G/T)^T.
\]
Here we fix  representatives $w \in W_G$ of all classes $[w] \in W_G/W_K$ once and for all. 

Consider the manifold
\[
M := G/T \times W_G/W_K
\]
equipped with the action by $G$ on the first factor.
Consider the $G$-invariant complex structure on $M$ such that, for every $[w] \in W_G/W_K$, the restricted complex structure on $G/T\times \{[w]\}$ is $J^{G/T}_{w^{-1}R^+}$. 

We have $T$-invariant neighbourhoods
\[
V := \bigcup_{w_K \in W_K} w_K U
\]
of $(G/T)^T$ and 
\[
V_c := \bigcup_{w \in W_G} wU_c^w
\]
of $(G_c/T)^T$. We choose the set $\tilde U$ small enough, the sets $w_K U$ and $wU_c^w$ are all disjoint, for $w_K \in W_K$ and $w \in W_G$.  Consider the map
\[
\varphi\colon V \times W_G/W_K \to V_c
\]
given by
\beq{eq def phi}
\varphi(w_K xT, [w]) = w_K w\varphi_w(xT),
\eeq
for $w_K \in W_K$, $w \in W_G$ and $xT \in U$. 
\begin{lemma}
The map $\varphi$ is well-defined and $T$-equivariant.
\end{lemma}
\begin{proof}
Let $w_K, w_K'\in N_K(T)$ and $x,x'\in G$ be such that $w_K xT = w_K' x'T$. Fix $w\in W_G$. Since the sets $w_K U$ and $w_K'U$ are disjoint if $w_K$ and $w_K'$ represent different elements of $W_K$, there is a $t\in T$ such that $w_K' = w_K t$. There is another $t' \in T$ such that $w_K' x' = w_K x t'$. Using $T$-equivariance of $w\varphi_w$ and the fact that $w_K$ normalises $T$, one deduces that   $w_K' w\varphi_w(x'T) = w_K w\varphi_w(xT)$, and that $\varphi$ is $T$-equivariant.
\end{proof}
 The arguments in this section lead to the following conclusion.
\begin{proposition}\label{prop phi}
The map $\varphi$ is a $T$-equivariant diffeomorphism, holomorphic up to a homotopy of complex structures, from a $T$-invariant neighbourhood of $M^T$ onto a $T$-invariant neighbourhood of $(G_c/T)^T$.
\end{proposition}
See Figure \ref{fig SL2} for an example of the map $\varphi$.

\begin{remark}
If one replaces the map $\varphi_w$ by $\varphi_e$ in \eqref{eq def phi}, one obtains an identification $\tilde \varphi$ of neighbourhoods of $M^T$ and $(G_c/T)^T$ that is simpler than the map $\varphi$ we use here. However, $\tilde \varphi$ is not $T$-equivariant, which is the reason why we use the maps $\varphi_w$ rather than $\varphi_e$. At the same time, using the maps $\varphi_w$ means that the map $\varphi$ is holomorphic up to homotopy with respect to the complex structure on $M$ described above, whereas
 the map  $\tilde \varphi$ is holomorphic up to homotopy if one uses the same complex structure $J^{G/T}_{R^+}$ on all connected components of $M$.
\end{remark}

\subsection{Line bundles}

Consider the line bundle $L \to M$ such that for every $[w] \in W_G/W_K$,
\[
L|_{G/T \times \{[w]\}} = G\times_{T}\C_{w^{-1}(\lambda - \rho)}.
\]
\begin{lemma}\label{lem L}
There is a $T$-equivariant isomorphism of line bundles
\[
\varphi^*(\LGc|_{V_c}) \cong L|_{V\times W_G/W_K}.
\]
\end{lemma}


\begin{lemma}\label{lem vb x}
Let $G$ be any Lie group, $H<G$ a compact subgroup, and $\pi\colon H \to \GL(V)$ a representation of $H$ in a finite-dimensional vector space $V$. Let $x \in N_G(H)$. Consider the representation $x^{-1}\cdot\pi\colon H \to \GL(V)$ given by $(x^{-1}\cdot \pi)(h) = \pi(x^{-1}hx)$. Write $V_{\pi}$ and $V_{x^{-1}\cdot \pi}$ for the vector space $V$ on which $H$ acts via $\pi$ and $x^{-1}\cdot \pi$, respectively. Fix an $\Ad(H)$-invariant inner product on $\kg$. Let $\kh^{\perp}$ be the orthogonal complement to $\kh$ in $\kg$. Define the map $\eta_x\colon \kh^{\perp} \to G/H$ by $\eta_x(X) = \exp(X)xH$, for $X \in \kh^{\perp}$. Let $\tilde U \subset \kh^{\perp}$ be an $\Ad(H)$-invariant open subset such that $\eta|_{\tilde U}$ is a diffeomorphism onto its image $U_x$. Then we have an $H$-equivariant isomorphism of vector bundles
\[
\eta_x^*((G\times_H V_{\pi})|_{U_x}) \cong \tilde U \times V_{x^{-1}\cdot \pi}.
\]
\end{lemma}
\begin{proof}
Note that 
\[
\eta_x^*((G\times_H V_{\pi})|_{U_x}) = \{ (X, [\exp(X)x, v]); X \in \tilde U,  v \in V_{\pi}\}. 
\]
Define the map 
\[
f\colon \eta_x^*((G\times_H V_{\pi})|_{U_x}) \to \tilde U \times V_{x^{-1}\cdot \pi}
\]
by
\[
f(X, [\exp(X)x, v]) = (X,v)
\]
for $X\in \tilde U$ and $v\in V_{\pi}$.
Note that this map is well-defined, and a vector bundle isomorphism.
And  for all $h \in H$,
\[
\begin{split}
h\cdot (X, [\exp(X)x, v]) &= (\Ad(h)X, [h\exp(X)x, v]) \\
&= (\Ad(h)X, [\exp(\Ad(h)X)x(x^{-1}hx), v]) \\
&=  (\Ad(h)X, [\exp(\Ad(h)X)x, (x^{-1}\cdot\pi)(h)v]).
\end{split}
\]
So
\[
f(h\cdot (X, [\exp(X)x, v])) = (\Ad(h)X, (x^{-1}\cdot \pi)(h)v) = h\cdot f(X, [\exp(X)x, v]).
\]
\end{proof}

\begin{proofof}{Lemma \ref{lem L}}
Let $w\in N_{G_c}(T)$. The claim is that we have a $T$-equivariant isomorphism of line bundles
\beq{eq line bdles iso}
(w\circ \varphi_w)^*(G_c\times_T \C_{\lambda - \rho}|_{wU^w_c}) \cong (G\times_T \C_{w^{-1}(\lambda - \rho)}|_U).
\eeq 
By Lemma \ref{lem vb x}, we have isomorphisms of $T$-equivariant line bundles
\beq{eq line bdles}
\begin{split}
\eta^*((G\times_T \C_{w^{-1}(\lambda - \rho}))|_{U}) &\cong \tilde U \times \C_{w^{-1}(\lambda - \rho)};\\
(\eta_c^w)^*((G_c\times_T \C_{\lambda - \rho}))|_{wU_c^w}) &\cong (\eta_c^w)^{-1}(wU_c^w) \times \C_{w^{-1}(\lambda - \rho)}.
\end{split}
\eeq
Now 
\[
(\eta_c^w)^{-1}(wU_c^w) = (\eta_c^w)^{-1}(w\eta_c(\psi_w(\tilde U))) = \Ad(w)\circ \psi_w(\tilde U).
\]
By Lemma \ref{lem Adw psi equivar}, we have a $T$-equivariant, holomorphic diffeomorphism
\[
\Ad(w)\circ \psi_w\colon \tilde U \xrightarrow{\cong} (\eta_c^w)^{-1}(wU_c^w).
\]
So the two right hand sides of \eqref{eq line bdles} are isomorphic as $T$-equivariant line bundles.

We saw in the proof of Lemma \ref{lem T equivar} that
\[
w\circ \varphi_w \circ \eta= \eta_c^w\circ (\Ad(w)\circ \psi_w).
\] 
So the pullback of the left hand side of \eqref{eq line bdles iso} along $\eta$ equals
\[
\begin{split}
(\Ad(w)\circ\psi_w)^* (\eta_c^w)^*((G_c\times_T \C_{\lambda - \rho}))|_{wU_c^w}) &= (\Ad(w)\circ\psi_w)^*\bigl((\eta_c^w)^{-1}(wU_c^w) \times \C_{w^{-1}(\lambda - \rho)}\bigr)\\
&= (\eta_c^w\circ \Ad(w)\circ\psi_w)^{-1}(wU_c^w)\times  \C_{w^{-1}(\lambda - \rho)}.\\
& = (w\circ \eta_c \circ\psi_w)^{-1}(wU_c^w) \times  \C_{w^{-1}(\lambda - \rho)}\\
&= \tilde U \times  \C_{w^{-1}(\lambda - \rho)}.
\end{split}
\]
The latter line bundle is 
the pullback of the right hand side of \eqref{eq line bdles iso} along $\eta$.
%
%
%
%
\end{proofof}


\subsection{Proof of Theorem \ref{thm char id}} \label{sec pf thm}

Since $G_c$ is compact, the claim is that if $\phi$ is the $L$-parameter of $\pi^G_{\lambda}$, and $\phi'$ is the corresponding $L$-parameter for $G_c$,
\beq{eq char id}
(-1)^{\dim(G/K)/2}\sum_{\pi \in \Pi_{\phi}} \Theta_{\pi}|_{T^{\reg}} = \sum_{\pi' \in \Pi_{\phi'}} \Theta_{\pi'}|_{T^{\reg}}.
\eeq

Let $g \in T$ be an element whose powers are dense in $T$.
Let $\bar \partial_L$ be the Dolbeault operator on $M$ coupled to $L$. 
Propositions \ref{prop loc} and \ref{prop phi} and Lemma \ref{lem L} imply that
\begin{multline}\label{eq pf 1}
\tau_g(\ind_{G_c}(\bar \partial_{\LGc} + \bar \partial_{\LGc}^*)) = \tau_g(\ind_G(\bar \partial_{L} + \bar \partial_{L}^*))\\
	= \sum_{[w] \in W_G/W_K} \tau_g(\ind_G(\bar \partial_{L^G_{w^{-1}(\lambda - \rho)}} + \bar \partial_{L^G_{w^{-1}(\lambda - \rho)}}^*)).
\end{multline}

For any regular element $\nu \in i\kt^*$, let $\rho_{\nu}$ be half the sum of the roots in $R$ with positive inner products with $\nu$. Then for any $w \in W_G$, $\rho_{w\lambda} = w\rho$. So the right hand side of \eqref{eq pf 1} equals
\[
 \sum_{[w] \in W_G/W_K} \tau_g(\ind(\bar \partial_{L^G_{w^{-1}\lambda - \rho_{w^{-1}\lambda}}} + \bar \partial_{L^G_{w^{-1}\lambda - \rho_{w^{-1}\lambda}}}^*)).
\]
By Proposition \ref{prop ds index}, this equals
\[
(-1)^{\dim(G/K)/2} \sum_{[w] \in W_G/W_K} \Theta_{\pi^G_{w^{-1}\lambda}}(g).
\]

As noted below Proposition \ref{prop Lpack ds}, we have
\[
\Pi_{\phi} = \{\pi^{G}_{w \lambda}; w \in W_G\}/W_K,
\]
and similarly,
\[
\Pi_{\phi'} = \{\pi^{G_c}_{\lambda}\}.
\]
Note that $\pi^{G_c}_{\lambda}$ is the irreducible representation of $G_c$ with infinitesimal character $\lambda$, so with highest weight $\lambda - \rho$.
The Borel--Weil--Bott theorem therefore implies that  the right hand side of \eqref{eq char id}, evaluated at $g$, equals
\[
\Theta_{\pi^{G_c}_{\lambda}}(g) = \tau_g(\ind_{G_c}(\bar \partial_{\LGc} + \bar \partial_{\LGc}^*)), 
\]
which by the above considerations equals the left hand side of \eqref{eq char id}, evaluated at $g$. We have assumed that the powers of $g$ are dense in $T$, but as such elements are dense in $T^{\reg}$, Theorem \ref{thm char id} follows in the case where $G' = G_c$ since both sides of \eqref{eq char id} are analytic on $T^{\reg}$.
As noted at the start of Section \ref{sec proof}, that case implies the general case, so the theorem has been proved. 

\subsection{Example: $G = \SL(2,\R)$}

Let $G=\SL(2,\R)$, and $G'=G_c = \SU(2)$. Note that $G$ is quasi-split. Take $K=T = \SO(2) \cong \U(1) \hookrightarrow \SU(2)$.   Let $\rho \in i\kt^*$ be the element mapping $X := \begin{bmatrix} {0} & {-1} \\{1} &{0} \end{bmatrix} \in \kt$ to $i$. Let $n \in \N$, and set $\lambda = n\rho$. 
Let $\pi^{G}_{\pm\lambda}$ be the discrete series representation of $\SL(2, \R)$ with Harish-Chandra parameter $\pm\lambda$.
Let $\phi$ be the $L$-parameter of  $\pi^G_{\lambda}$; see Example \ref{ex L param SL2}. Then $\Pi_{\phi}= \{\pi^{G}_{\lambda}, \pi^{G}_{-\lambda}\}$.
 Let $\phi'$ be the corresponding $L$-parameter for $G_c$. Then
$\Pi_{\phi'}= \{\pi^{G_c}_{\lambda}\}$, where $\pi^{G_c}_{\lambda}$ is the irreducible representation of $\SU(2)$ with infinitesimal character $\lambda$, so with highest weight $\lambda - \rho$. 

Fix $g \in T$ with dense powers. Now $G_c/T = S^2$, and $(G_c/T)^g$ consists of the north and south poles. And $G/T$ is the hyperbolic plane $\{x^2+y^2-z^2=n\}$, and $(G/T)^g$ is the single point $eT = (0,0,n)$. We have $W_K = \{e\}$ and $W_G = \Z_2$, where the nontrivial element of $W_G$ is represented by\footnote{Note that as matrices, we have $w = X$. We use different letters because $w$ is viewed as a group element, whereas $X$ is viewed as a Lie algebra element.
} $w:=\begin{bmatrix} {0} & {-1} \\{1} &{0} \end{bmatrix} \in N_{\SU(2)}(\U(1))$. This element maps the north pole of $G_c/T = S^2$ to the south pole. So we see that indeed
\[
(G_c/T)^g = (G/T)^g \cup w (G/T)^g,
\] 
The map $\varphi$ identifying neighbourhoods of these fixed point sets in $G_c/T$ and $G/T \times W_G/W_K$, respectively, is depicted in 
 Figure \ref{fig SL2}.

The Atiyah--Segal--Singer or Atiyah--Bott fixed point formula implies that, if $g = \exp(tX)$,
\beq{eq index SU2}
\ind_{\SU(2)}(\bar \partial_{L^{\SU(2)}_{\lambda - \rho}} + \bar \partial_{L^{\SU(2)}_{\lambda - \rho}}^*) (g) = \frac{e^{\lambda}(g) - e^{-\lambda}(g)}{e^{\rho}(g) - e^{-\rho}(g)} = \frac{\sin(nt)}{\sin(t)}.
\eeq
Theorem 2.1 in \cite{Hochs-Wang} implies that
\[
\tau_g(\ind_{\SL(2,\R)} (\bar \partial_{ L^{\SL(2,\R)}_{\pm(\lambda - \rho)} } 
+ \bar \partial_{L^{\SL(2,\R)}_{\pm(\lambda - \rho)} }^*)) = \frac{e^{\pm\lambda}(g) }{\pm(e^{\rho}(g) - e^{-\rho}(g))} = \pm\frac{e^{\pm int}}{2i\sin(t)}.
\]

Let $\phi$ be the $L$-parameter of $\pi^G_{\lambda}$, and let $\phi'$ be the corresponding $L$-parameter for $\SU(2)$. Then we conclude that
\begin{multline*}
(-1)^{\dim(\SL(2,\R)/\SO(2))/2}\sum_{\pi \in \Pi_{\phi}} \Theta_{\pi}(g) = \\
 \tau_g(\ind_{\SL(2,\R)}(\bar \partial_{L^{\SL(2,\R)}_{\lambda - \rho}} + \bar \partial_{L^{\SL(2,\R)}_{\lambda - \rho}}^*))  
+\tau_g( \ind_{\SL(2,\R)}(\bar \partial_{L^{\SL(2,\R)}_{-(\lambda - \rho)}} + \bar \partial_{L^{\SL(2,\R)}_{-(\lambda - \rho)}}^*)) 
= \\
\ind_{\SU(2)}(\bar \partial_{L^{\SU(2)}_{\lambda - \rho}} + \bar \partial_{L^{\SU(2)}_{\lambda - \rho}}^*) (g)
=\\ (-1)^{\dim(\SU(2)/\SU(2))/2}
 \sum_{\pi' \in \Pi_{\phi'}} \Theta_{\pi'}(g).
\end{multline*}

In this example, we see from the fixed point formulas used, and Weyl's and Harish-Chandra's character formulas, that the respective  indices evaluated at $g$ equal the characters of the corresponding representations. And the character identity follows directly from the explicit expressions for these characters. But note that in the proof of Theorem \ref{thm char id}, these fixed point formulas and character formulas were not used.

\begin{figure}
\begin{centering}
\includegraphics[width=0.95\textwidth]{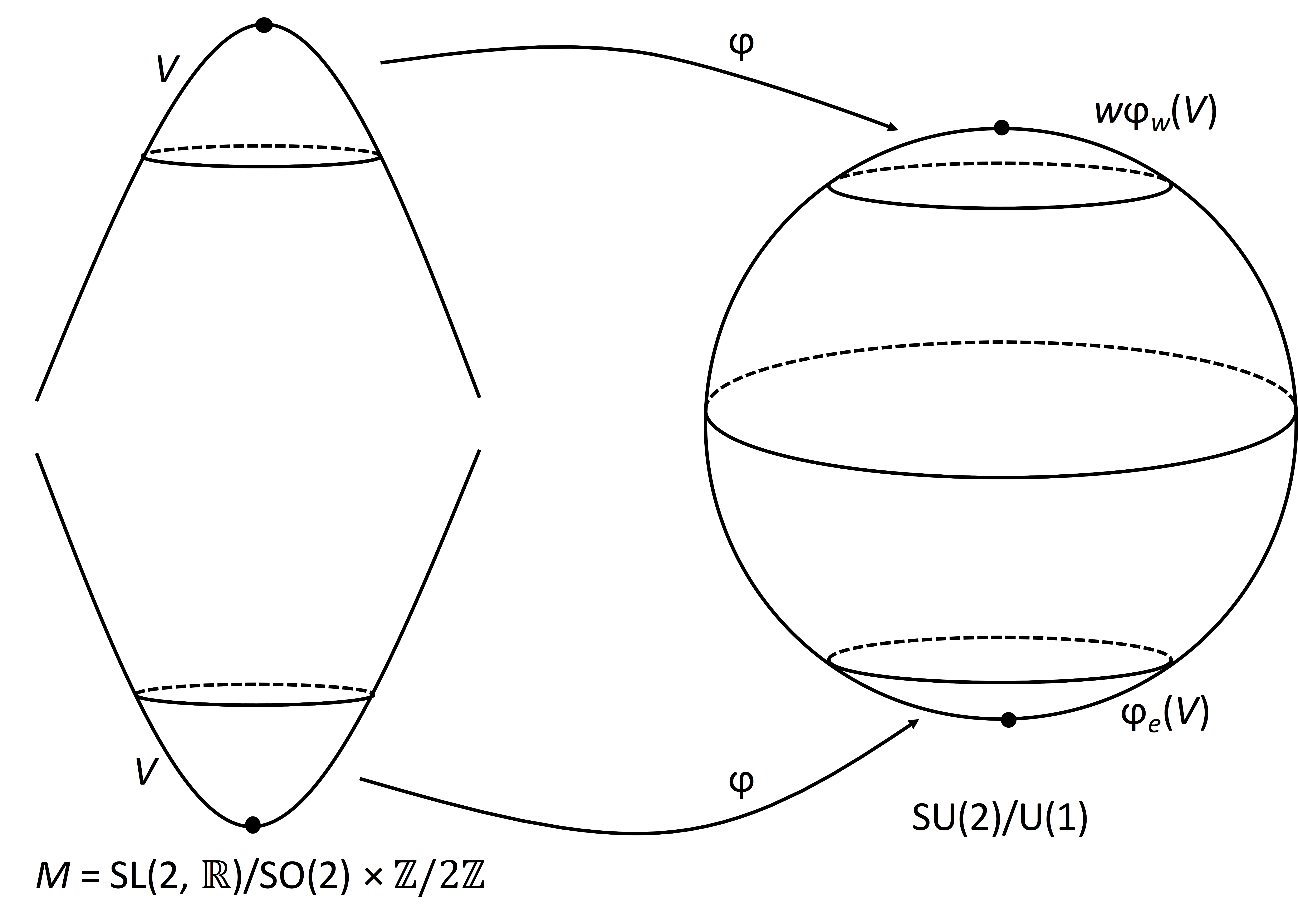}
\caption{The map $\varphi$ for $G = \SL(2,\R)$}
\label{fig SL2}
\end{centering}
\end{figure}

\end{document}